\documentclass[11pt,  a4paper]{amsart}
\usepackage{eucal}
\usepackage[english]{babel}
\usepackage{enumerate}
\usepackage[T1]{fontenc} 
\usepackage[utf8]{inputenc}
\usepackage[bookmarks=true,pdfstartview=FitH, pdfborder={0 0 0}, colorlinks=true,citecolor=red, linkcolor=blue]{hyperref}
\usepackage{bbm}
\usepackage{aliascnt}
\usepackage{a4wide}

\newcommand {\R}	{\mathbb{R}}
\newcommand {\N}	{\mathbb{N}}
\newcommand {\C}	{\mathbb{C}}

\DeclareMathOperator{\Id}{Id}
\DeclareMathOperator{\rg}{rg}
\newcommand{\p}{{\raisebox{1.3pt}{{$\scriptscriptstyle\bullet$}}}}

\newcommand{\dX}{{\partial X}}
\newcommand{\sA}{\mathcal{A}}
\newcommand{\sB}{\mathcal{B}}
\newcommand{\sC}{\mathcal{C}}
\newcommand{\sX}{\mathcal{X}}
\newcommand{\sP}{\mathcal{P}}
\newcommand{\sL}{\mathcal{L}}
\newcommand{\sT}{\mathcal{T}}
\newcommand{\rC}{\mathrm{C}}
\newcommand{\rL}{\mathrm{L}}
\newcommand{\COq}{\rC(\overline{\Omega})}
\newcommand{\dO}{{\partial\Omega}}
\newcommand{\CdO}{\rC(\dO)}
\newcommand{\ddn}{\frac{\partial}{\partial n}}
\newcommand{\tddn}{\tfrac{\partial}{\partial n}}
\newcommand{\Gnn}{G_{00}}
\newcommand{\Ann}{A_{00}}

\newcommand{\CnO}{\rC_0(\Omega)}
\newcommand{\Deltam}{\Delta_m}
\newcommand{\eps}{\varepsilon}
\newcommand{\imp}{\,\Rightarrow\,}

\newcommand{\diag}{\operatorname{diag}}
\newcommand{\MnC}{\mathrm{M}_n(\C)}
\renewcommand{\epsilon}{\varepsilon}
\newcommand{\DLB}{\Delta_{\Gamma}}

\advance\textheight by 2cm

\theoremstyle{plain}
\newtheorem{thm}{Theorem}[section]
\newaliascnt{cor}{thm}
\newaliascnt{prop}{thm}
\newaliascnt{lem}{thm}
\newtheorem{cor}[cor]{Corollary}

\newtheorem{prop}[prop]{Proposition}
\newtheorem{lem}[lem]{Lemma}
\aliascntresetthe{cor}
\aliascntresetthe{prop}
\aliascntresetthe{lem} 
\newcounter{stp}
\newcounter{stpi}
\newcounter{stpii}
\newtheorem{step}[stp]{Step}
\newtheorem{stepi}[stpi]{Step}
\newtheorem{stepii}[stpii]{Step}
\theoremstyle{definition}
\newaliascnt{defn}{thm}
\newaliascnt{asu}{thm}
\newaliascnt{con}{thm}
\aliascntresetthe{defn}
\aliascntresetthe{asu}
\aliascntresetthe{con}
\theoremstyle{remark}
\newaliascnt{rem}{thm}
\newaliascnt{exa}{thm}
\newaliascnt{masu}{thm}
\newaliascnt{nota}{thm}
\newaliascnt{sett}{thm}
\newtheorem{rem}[rem]{Remark}

\newtheorem{masu}[masu]{Assumptions}
\newtheorem{nota}[nota]{Notation}
\newtheorem{sett}[sett]{Abstract Setting}
\aliascntresetthe{rem}
\aliascntresetthe{exa}
\aliascntresetthe{masu}
\aliascntresetthe{nota}
\aliascntresetthe{sett}
\newenvironment{psmallmatrix}
{\left(\begin{smallmatrix}}
{\end{smallmatrix}\right)}
\numberwithin{equation}{section}

\title[Wentzell boundary conditions and the Dirichlet-to-Neumann operator]
{Operators with Wentzell boundary conditions and the Dirichlet-to-Neumann operator}
\author{Tim Binz}
\author{Klaus-Jochen Engel}

\date{\today}%

\begin{document}
\renewcommand{\sectionautorefname}{Section}
\renewcommand{\subsectionautorefname}{Subsection}

\begin{abstract}
In this paper we relate the generator property of an operator $A$ with (abstract) generalized Wentzell boundary conditions on a Banach space $X$ and its associated (abstract) Dirichlet-to-Neumann operator $N$ acting on a ``boundary'' space $\dX$.
Our approach is based on similarity transformations and perturbation arguments
and allows to split $A$ into an operator $\Ann$ with Dirichlet-type boundary conditions on a space $X_0$ of states having ``zero trace'' and the operator $N$. If $\Ann$ generates an analytic semigroup, we obtain under a weak Hille--Yosida type condition that $A$ generates an analytic semigroup on $X$ if and only if $N$ does so on $\dX$.
Here we assume that the (abstract) ``trace'' operator $L:X\to\dX$ is bounded what is typically satisfied if $X$ is a space of continuous functions. Concrete applications are made to various second order differential operators.

\end{abstract}
\maketitle

\section{Introduction}
\label{sec:intro}

The generation of analytic semigroups by differential operators with generalized Wentzell boundary conditions on spaces of continuous functions attracted the interest of many authors, and we refer, e.g., to \cite{CM:98}, \cite{FGGR:02}, \cite{Eng:03}, 
\cite{EF:05}, \cite{FGGR:10}. For their derivation and physical interpretation we refer to \cite{GGol:06}.
The present paper is a continuation and improvement of \cite{EF:05} where we introduced a general abstract framework to deal with this problem. Before recalling this setting we consider the following typical example in order to explain the basic ideas and the goal of our approach.

\smallskip
Take a smooth bounded domain $\Omega\subset\R^n$. Then consider on $\COq$ the Laplacian $\Deltam$ with ``maximal'' domain $D(\Deltam):=\{f\in \COq:\Deltam f\in\COq\}$, where the derivatives are taken in the distributional sense. Finally, let $\ddn:D(\ddn)\subset\COq\to\CdO$ be the outer normal derivative, $\beta<0$ and $\gamma\in\CdO$. In this setting we define the Laplacian $A\subset\Deltam$ with \emph{generalized Wentzell boundary conditions} by requiring
\begin{equation}\label{eq:bc-W-Lap} 
f\in D(A)
\quad:\iff\quad
\Deltam f\big|_{\partial\Omega}=\beta\cdot\tddn\; f+\gamma\cdot f\big|_{\partial\Omega}.
\end{equation}
Our approach decomposes a function $f\in\COq$ into the (unique) sum $f=f_0+h$ of a function $f_0$ vanishing at the boundary $\dO$ and a harmonic function $h$ having the same trace as $f$. In other words, if $L:\COq\to\CdO$, $Lf:=f|_\dO$ denotes the trace operator, then $f_0\in\ker L=\CnO$ while $h\in\ker(\Deltam)$. Since $h$ is uniquely determined by its trace, it can be identified with its boundary value $x:=Lh$. Hence, every $f\in\COq$ corresponds to a unique pair $\binom{f_0}{x}\in\CnO\times\CdO$.

\smallskip
To formalize this decomposition we introduce an abstract ``\emph{Dirichlet operator}'' $L_0:\CdO\to\COq$. To this end we consider for a given ``boundary function'' $x\in\CdO$ 
the Dirichlet problem
\begin{equation}\label{eq:Dir-Prob}
\begin{cases}
\Deltam f=0,\\
f|_\dO=x.
\end{cases}
\end{equation}
This system admits a unique solution $f\in\COq$, so by setting $L_0x:=f$ we obtain an operator $L_0\in\sL(\CdO,\COq)$. For $f\in\COq$ we then have $f=f_0+h$ where $f_0:=(\Id-L_0L)f$ and $h=L_0x$ for $x:=Lf$. By \eqref{eq:bc-W-Lap} it then follows (for the details see \autoref{ÄT} below in the proof of \autoref{Maintheorem}) that $A$ on $\COq$ transforms into an operator matrix $\sA$ on $\CnO\times\CdO$ of the form
\begin{equation*}
\sA:=
\begin{pmatrix}
\Delta_m&0\\
0&N
\end{pmatrix}+
\sP
\end{equation*}
with some appropriate ``non-diagonal'' domain $D(\sA)\subset\CnO\times\CdO$, see \cite{Eng:98}, \cite{Eng:99}, \cite{Nag:90}. Here $\sP$ denotes an unbounded perturbation while $N:=\beta\cdot\ddn\cdot L_0$ is the so called \emph{Dirichlet-to-Neumann operator} on $\CdO$, see \cite{Esc:94}, \cite[Sect.~12.C]{Tay:96b}. That is, $Nx$ is obtained by applying the Neumann boundary operator to the solution $f$ of the Dirichlet problem \eqref{eq:Dir-Prob}. 

\smallskip
Using perturbation arguments one can show that $\sA$, hence also $A$, generate analytic semigroups if and only if the Dirichlet Laplacian $\Delta_{00}$ on $\CnO$ and the Dirichlet-to-Neumann operator $N$ on $\CdO$ do so. This means that we  decoupled the operator $A\subset\Deltam$ with generalized Wentzell boundary conditions on $X:=\COq$ into an operator $\Ann:=\Delta_{00}$ with Dirichlet boundary conditions on $X_0:=\CnO$ and the Dirichlet-to-Neumann operator $N:=\beta\cdot\ddn\cdot L_0$ on the boundary space $\dX:=\CdO$. 

\smallskip
Since it is well-known that $\Delta_{00}$ generates an analytic semigroup, our main result applied to this example yields that $A$ generates an analytic semigroup on $\COq$ if and only if $N$ generates  an analytic semigroup on $\CdO$. Since the latter is true, see \cite[Sect.~2]{Eng:03}, we conclude that $A\subset\Deltam$ with generalized Wentzell boundary condition \eqref{eq:bc-W-Lap} is the generator of an analytic semigroup. We mention that our approach also keeps track of the angle of analyticity and, in the above example, gives the optimal angle $\frac{\pi}{2}$.

\smallskip
This paper is organized as follows. In \autoref{AS} we introduce our abstract setting and then state in \autoref{MR} our main abstract generation result, \autoref{Maintheorem}. In the following \autoref{Pert} we show that the generator property of operators with generalized Wentzell boundary conditions is invariant under ``small'' perturbations with respect to the action as well as the domain, cf. \autoref{Stoerung} and \autoref{Stoerung2}. For these proofs we study in \autoref{Lemma 2} and  \autoref{Stoerungssatz für D-N} how the Dirichlet- and Dirichlet-to-Neumann operator, respectively, behaves under relatively bounded perturbations.
Finally, in \autoref{Expl} we apply our abstract results to second order differential operators on $\rC([0,1],\C^n)$, the Banach space-valued second-order derivative, a perturbed Laplacian with generalized Wentzell boundary conditions and uniformly elliptic operators on $\rC(\overline{\Omega})$. Our notation follows the monograph \cite{EN:00}.

\section{The Abstract Setting}\label{AS}

As in \cite[Section 2]{EF:05}, the starting point of our investigation is the following

\begin{sett}\label{set:AS}
Consider
\begin{enumerate}[(i)]
\item two Banach spaces $X$ and $\dX$, called \emph{state} and 
\emph{boundary space}, respectively;
\item a densely defined \emph{maximal operator}
$A_m \colon D(A_m) \subset X \rightarrow X$;
\item a \emph{boundary (or trace) operator} $L \in \sL(X,\dX)$;
\item a \emph{feedback operator} $B \colon D(B) \subseteq X \rightarrow \dX$.
\end{enumerate}
\end{sett}

Using these spaces and operators we define the operator $A^B:D(A^B)\subset X\to X$
with abstract \emph{generalized Wentzell boundary conditions} by
\begin{equation}\label{eq:W-BC}
A^B \subseteq A_m, \quad 
D(A^B):= \bigl\{ f \in D(A_m) \cap D(B) : LA_mf = Bf \bigr\} . 
\end{equation}
If $B=0$ the boundary conditions defined by \eqref{eq:W-BC} are called \emph{pure Wentzell boundary conditions}. For an interpretation of Wentzell- as ``dynamic boundary conditions'' we refer to \cite[Sect.~2]{EF:05}.

\smallskip
To fit the example from the introduction into this setting it suffices to choose $X:=\COq$, $\dX:=\CdO$, $A_m:=\Deltam$, $Lf:=f|_{\dO}$ and $B:=\beta\cdot\ddn+\gamma\cdot L$.

\smallskip
In the sequel we need the (in general non-densely defined) operator
$A_0:D(A_0)\subset X\to X$ defined by
\begin{equation*}
A_0\subseteq A_m,
\quad
D(A_0):=D(A_m)\cap\ker(L).
\end{equation*}
In the example from the introduction $A_0$ is the Dirichlet Laplacian $\Delta_0$ on $\COq$ with non-dense domain $D(A_0)=D(\Delta_m)\cap\CnO$.

\begin{masu}
\makeatletter
\hyper@anchor{\@currentHref}%
\makeatother
\label{masu}
\begin{enumerate}[(i)]
\item The operator $A_0$ is a weak Hille--Yosida operator on $X$, i.e. 
there exist $\lambda_0\in \R$ and $M > 0$ such that $[\lambda_0,\infty) \subset \rho(A_0)$ and 
\begin{align*}
\bigl\| \lambda R(\lambda,A_0)\bigr\| \leq M
\quad\text{for all $\lambda \geq \lambda_0$;}
\end{align*}
\item the operator $B$ is relatively $A_0$-bounded with bound $0$, i.e., $D(A_0)\subseteq D(B)$ and for every $\eps>0$ there exists $M_\eps>0$ such that
\begin{equation*}
\|Bf\|_\dX\le \eps\cdot\|A_0f\|_X+M_\eps\cdot\|f\|_X
\quad\text{for all }f\in D(A_0);
\end{equation*}
\item the \emph{abstract Dirichlet operator} $L_0 := (L|_{\ker(A_m)})^{-1} : \dX \rightarrow \ker(A_m)\subseteq X$ exists and is bounded, i.e., for every $x\in\dX$ the abstract Dirichlet problem
\begin{equation*}
\begin{cases}
A_m f = 0, \\
\ Lf = x
\end{cases}
\end{equation*}
admits a unique solution $f\in D(A_m)$ and $L_0 x:=f$ defines an operator $L_0\in\sL(\dX,X)$.
\end{enumerate}
\end{masu}
We note that by \cite[Lem.~1.2]{Gre:87} assumption~(iii) is always satisfied if $A_m$ is closed, $L:X\to\dX$ is surjective and $A_0$ is invertible. Moreover, $L_0L\in\sL(X)$ is a projection onto the subspace $\ker(A_m)$ along $X_0 := \ker(L)$ which induces the decompositions
\begin{equation}\label{eq:split-X}
X=X_0\oplus \ker(A_m)
\qquad\text{and}\qquad
D(A_m)=D(A_0)\oplus\ker(A_m).
\end{equation}

\smallskip
In the sequel we will need  the following operators. 

\begin{nota} Define $G_m:D(G_m)\subset X\to X$ by 
\begin{equation*}
G_mf:=A_mf-L_0B\cdot(\Id-L_0L)f,
\quad
D(G_m):=D(A_m).
\end{equation*}
Then for $*\in\{1,0,00\}$ we consider the restrictions $A_*\subset A_m$ and $G_*\subset G_m$
given by
\begin{alignat*}{3}
&A_{0}:D(A_0)\subset X\to X,&&D(A_{0}) := \{ f \in D(A_m) : Lf = 0 \}, \\
&A_{1}:D(A_1)\subset X\to X,&\quad&D(A_{1}) := \{ f \in D(A_m) : LA_mf = 0 \},\\
&\Ann:D(\Ann)\subset X_0\to X_0,&&D(\Ann) := \{ f \in D(A_m): Lf=0,\;LA_mf = 0 \}\\
\intertext{and}
&G_{0}:D(G_0)\subset X\to X,&&D(G_{0}) := D(A_0),\\
&G_{1}:D(G_1)\subset X\to X,&\quad&D(G_{1}) := \{ f \in D(G_m) : LG_mf = 0 \},\\
&G_{00}:D(G_{00})\subset X_0\to X_0,&&D(G_{00}) := \{ f \in D(G_m): Lf=0,\;LG_mf = 0 \}.
\end{alignat*}
Observe that $G_{00}\subset G_0= A_0 - L_0 B$. In other words, $D_*$ for $D\in\{A,G\}$ and $*\in\{0,1,00\}$ is a restriction of $D_m$. For $*=0$ this restriction corresponds to abstract Dirichlet boundary conditions and for $*=1$ to pure Wentzell boundary conditions on $X$, while $D_{00}$ is the part of $D_0$ as well as of $D_1$ in $X_0$.

\smallskip
Finally, we define the abstract \emph{Dirichlet-to-Neumann operator} $N:D(N)\subset\dX\to\dX$ by
\begin{equation*}
Nx:=BL_0x,\quad
D(N):=\bigl\{x\in\dX:L_0x\in D(B)\bigr\}.
\end{equation*}
This operator plays a crucial role in our approach.
\end{nota}

\section{The Main Result}\label{MR}

The following is our main abstract result. In contrast to \cite[Thm.~3.1]{EF:05} it proves (besides further generalizations) that (a)$\iff$(b) and not only that (b)$\imp$(a) in case $D=A$.

\begin{thm}\label{Maintheorem}
Let $D\in\{A,G\}$. Then the following statements are equivalent
\begin{enumerate}[(a)]
\item $A^B$ given by \eqref{eq:W-BC} generates an analytic semigroup of angle $\alpha > 0$ on $X$.
\item $D_0$ is sectorial of angle $\alpha > 0$ on $X$ and the Dirichlet-to-Neumann 
operator $N$ generates an analytic
semigroup of angle $\alpha > 0$ on $\dX$.
\item $D_1$ and $N$ generate analytic semigroups of angle 
$\alpha > 0$ on $X$ and $\dX$, respectively.
\item $D_{00}$ and $N$ generate analytic semigroups of angle 
$\alpha > 0$ on $X_0$ and $\dX$, respectively.
\end{enumerate} 
\end{thm}
\begin{proof}
By \cite[Thm.~3.1]{EF:05} we have that (b)$\imp$(a) for $D_0=A_0$. Since $A_0$ and $G_0$ only differ by a relatively bounded perturbation of bound $0$, \cite[Lem.~III.2.6]{EN:00} implies that assumption~(b) is equivalent for $D=A$ and $D=G$. This shows that  (b)$\imp$(a).
The equivalences (b)$\iff$(c)$\iff$(d) for $D=A$ follow by \cite[Lem.~3.3]{EF:05}. 
Now assume that  $D=G$. Then by \cite[Lem.~III.2.5]{EN:00}  there exists $\lambda\in\rho(G_0)$. Since $L$ is surjective, \cite[Lem.~1.2]{Gre:87} implies that the Dirichlet operator for $G_m-\lambda$ exists. As before, \cite[Lem.~3.3]{EF:05} now applied to $G_0-\lambda$, $G_1-\lambda$ and $G_{00}-\lambda$ gives the equivalence of (b), (c) and (d) for $D=G$.

\smallskip
To complete the proof it suffices to verify that (a)$\imp$(d) for $D_{00}=\Gnn$. We proceed in several steps where we put $\sX_0 := X_0 \times \dX$.

\begin{step}\label{ÄT}
The operator $A^B:D(A^B)\subset X\to X$ is similar to $\sA:D(\sA)\subset\sX_0\to\sX_0$ given by
\begin{align*}
\sA := \begin{pmatrix}
G_0 && -L_0N \\ B && N 
\end{pmatrix},
\quad
D(\sA) :=\Bigl\{ \tbinom fx \in D(A_0) \times D(N) \colon G_0 f -L_0 Nx \in X_0 \Bigr\}.
\end{align*}
\end{step}
\begin{proof}
The operator
\begin{align*}
T : X \rightarrow \sX_0 
,\quad 
Tf :=\tbinom{f - L_0 L f}{Lf}
\end{align*}
is bounded and invertible with bounded inverse
\begin{align*}
T^{-1} : \sX_0 \rightarrow X
,\quad
T^{-1}\tbinom fx= f + L_0 x \ .
\end{align*}
We show that $\sA=T A T^{-1}$. Using that $LL_0=\Id_\dX$, $X_0=\ker(L)$ and $A_mL_0=0$ we have
\begin{align*}
\tbinom fx
\in D(\sA)
&\iff
f \in D(A_0),\; x \in D(N) \text{ and } A_m f - L_0Bf - L_0Nx  \in X_0 \\
&\iff
f \in D(A_0),\; x \in D(N) \text{ and } LA_m f - Bf - Nx  = 0 \\
&\iff
f \in D(A_0),\; x \in D(N) \text{ and } LA_m (f + {L_0 x})
= B(f + L_0 x) \\
&\iff T^{-1} 
\tbinom fx  \in D(A)
\iff 
\tbinom fx \in TD(A).
\end{align*}
Moreover, for $\binom fx \in TD(A)=D(\sA)$ we obtain using that $f+L_0x\in D(A)$
\begin{align*}
T A T^{-1}
\tbinom fx
&=
T A_m (f + L_0 x) \\
&= \begin{pmatrix}
 A_m (f + L_0 x) - L_0 L  A_m (f + L_0 x) \\ L A_m (f + L_0 x)
\end{pmatrix} \\
&= \begin{pmatrix}
A_0 f - L_0 B f - L_0 N x  \\ 
Bf + N x 
\end{pmatrix} \\
&= \begin{pmatrix}
G_0 && -L_0 N \\ B && N 
\end{pmatrix}
\binom fx.
\qedhere
\end{align*}
\end{proof}

\begin{step}
The operator $\sA_0:D(\sA_0)\subset\sX_0\to\sX_0$ given by
\begin{align*}
\sA_0 := \begin{pmatrix}
G_0 && -L_0N \\ 0 && N 
\end{pmatrix},
\quad
D(\sA_0) :=D(\sA)
\end{align*}
generates an analytic semigroup of angle $\alpha>0$ on $\sX_0$.
\end{step}

\begin{proof} By assumption $A$ generates an analytic semigroup of angle $\alpha>0$ on $X$. Hence, by \autoref{ÄT}, $\sA$ generates an analytic semigroup of angle $\alpha>0$ on $\sX_0$.
Since $B$ is relatively $A_0$-bounded with bound zero, a simple computation using the triangle inequality shows that 
$\mathcal{B} := \begin{psmallmatrix}
0 && 0 \\
B && 0 
\end{psmallmatrix}$ with domain $D(\sB) := (D(B)\cap X_0) \times \dX$ is 
relatively $\sA$-bounded with bound zero. Hence, by \cite[Lemma III.2.6]{EN:00} also $\sA_0=\sA-\sB$ generates an analytic semigroup with angle $\alpha > 0$ on $\sX_0$. 
\end{proof}

\begin{step}
There exists $\lambda_0\in\R$ such that $[\lambda_0,+\infty)\subset\rho(G_0)\cap\rho(\Gnn)\cap\rho(N)$ and
\begin{equation}\label{eq:res-sA0}
R(\lambda,\sA_0)=
\begin{pmatrix}
R(\lambda,\Gnn)&-R(\lambda,G_0)L_0N R(\lambda,N)\\0&R(\lambda,N)
\end{pmatrix}
\quad\text{for $\lambda\ge\lambda_0$}.
\end{equation}
\end{step}

\begin{proof} By assumption $A_0$ is a weak Hille--Yosida operator. Since $A_0$ and $G_0=A_0-L_0B$ differ only by a relatively bounded perturbation of bound $0$, by \cite[Lem.~III.2.5]{EN:00} also $G_0$ is a weak Hille--Yosida operator. In particular, there exists $\lambda_0\in\R$ such that $[\lambda_0,+\infty)\subset\rho(G_0)\cap\rho(\sA_0)$. Moreover,  \cite[Prop.~IV.2.17]{EN:00} implies $\rho(G_0)=\rho(\Gnn)$ which shows the first claim.

\smallskip
Next we claim that $\lambda-N$ is injective for $\lambda\ge\lambda_0$. If by contradiction we assume that there exists $0\ne x\in\ker(\lambda-N)$, a simple computation shows that 
\begin{equation*}
0\ne\binom{-R(\lambda,G_0)L_0Nx}{x}\in\ker(\lambda-\sA_0)
\end{equation*}
contradicting the fact $\lambda\in\rho(\sA_0)$.
Let now $R(\lambda,\sA_0)=(R_{ij}(\lambda))_{2\times2}$ and choose some arbitrary $\binom gy\in\sX_0$. Then
we have
\begin{align}\notag
\binom{R_{11}(\lambda)g+R_{12}(\lambda)y}{R_{21}(\lambda)g+R_{22}(\lambda)y}
=\binom fx
&\iff
(\lambda - \sA_0) \tbinom f x
=\tbinom g  y\\
&\iff
\begin{cases}\label{eq:N-sur}
(\lambda - G_0)f + L_0 N x &= g \\
(\lambda - N)x &= y \\
LG_0 f &= Nx.
\end{cases}
\end{align}
For $y = 0$ it follows $(\lambda - N) x = 0$ and hence $x = 0$. This implies  $R_{21}(\lambda) = 0$.  Moreover, by \eqref{eq:N-sur} the operator $\lambda-N$ must be surjective, hence it is invertible with inverse $(\lambda-N)^{-1}=R_{22}(\lambda)\in\sL(\dX)$. Again by \eqref{eq:N-sur} this implies $R_{11}(\lambda)=R(\lambda,\Gnn)$. On the other hand, choosing $g=0$ we obtain $R_{21}(\lambda)=-R(\lambda,G_0)L_0N R(\lambda,N)$ as claimed.
\end{proof}

\begin{step}
$D_{00}$ and $N$ generate analytic semigroups of angle 
$\alpha > 0$ on $X_0$ and $\dX$, respectively.
\end{step}

\begin{proof} Denote by $(\sT_0(t)_{t\ge0}$ the semigroup generated by $\sA_0$. Then by \cite[Thm.~II.1.10]{EN:00} for $\lambda\in\R$ sufficiently large $R(\lambda,\sA_0)$ is given by the Laplace transform $(\sL\sT_0(\p))(\lambda)$ of $(\sT_0(t)_{t\ge0}$.
Since $\sL$ is injective, \eqref{eq:res-sA0} implies that the 
semigroup generated by $\sA_0$ is given by
\begin{align*}
\sT_0(t) = \begin{pmatrix}
T(t) && * \\ 0 && S(t)
\end{pmatrix} \ , 
\end{align*}  
where $(T(t))_{t \geq 0}$ and $(S(t))_{t \geq 0}$ are semigroups on $X_0$ and
$\dX$ generated by $\Gnn$ and $N$, respectively. Since by assumption $(\sT_0(t)_{t\ge0}$ is analytic of angle $\alpha>0$, also the semigroups generated by $\Gnn$ and $N$ are analytic of angle $\alpha$.
\end{proof}
This completes the proof of \autoref{Maintheorem}. 
\end{proof}

Since by \cite[Thm.~II.4.29]{EN:00} an analytic semigroup is compact if and only if its generator has compact resolvent, the following result relates compactness of the semigroups generated by $A$ and $D_{00}$, $N$.

\begin{cor}\label{kompakt}
Let $D\in\{A,G\}$. Then
$A$ has compact resolvent if and only if $D_{0}$ and $N$ have compact resolvents on $X$ and $\dX$, respectively.
\end{cor}
\begin{proof}
By \autoref{ÄT}, $A$ has compact resolvent if and only if $\sA$ has.
Since $\sA$ and $\sA_0$ differ only by the relatively bounded perturbation $\mathcal{B} := \begin{psmallmatrix}
0 && 0 \\
B && 0 
\end{psmallmatrix}$ of bound $0$, by \cite[III-(2.5)]{EN:00} one of the operators $\sA,\sA_0$  has compact resolvent if and only if the other has. Let $\lambda\in\rho(\sA_0)$. Then by \eqref{eq:res-sA0} $R(\lambda,\sA_0)$ is compact if and only if $R(\lambda,\Gnn)$, $R(\lambda,N)$ and
\begin{equation*}
-R(\lambda,G_0)L_0NR(\lambda,N)=
R(\lambda,G_0)L_0-\lambda R(\lambda,G_0)L_0R(\lambda,N)
\end{equation*}
are all compact. The latter is the case if and only if $R(\lambda,G_0)L_0$ is compact. Now writing
\begin{align*}
R(\lambda,G_0)=R(\lambda,\Gnn)\cdot(\Id-L_0L)+R(\lambda,G_0)L_0\cdot L
\end{align*}
we conclude that $R(\lambda,\sA_0)$ is compact if and only if $R(\lambda,G_0)$ and $R(\lambda,N)$ are compact.
\end{proof}

\section{Perturbations of Operators with Generalized Wentzell Boundary Conditions}\label{Pert}

%Next we give some perturbation results for operators with generalized Wentzell boundary conditions. 
In many applications the feedback operator $B:D(B)\subset X\to\dX$ which determines the boundary condition in \eqref{eq:W-BC} splits into a sum
\begin{equation}\label{split-B}
B=B_0+CL, \quad D(B) = D(B_0) \cap D(CL) 
\end{equation}
for some $C:D(C)\subset \dX \rightarrow \dX$. For example in \eqref{eq:bc-W-Lap} we could choose $B_0= \beta
\frac{\partial}{\partial n}$ (which determines the feedback from the interior of $\Omega$ to the boundary $\dO$) and the multiplication operator $C=M_\gamma \in \mathcal{L}(\dX)$ (which governs the ``free'' evolution on $\dO$). Next we study this situation in more detail where we allow $C$ to be unbounded. For a concrete example see \cite[(1.2), (3.3)]{FGGR:10} and \autoref{LgW}. Moreover, we will introduce a relatively bounded perturbation $P$ of the operator $A_m$.

\smallskip
To this end we first have to generalize our notation concerning the Dirichlet- and Dirichlet-to-Neumann operators.
For a closed operator $D_m:D(D_m)\subset X\to X$ let $D_0\subset D_m$ with domain $D(D_0):=D(D_m)\cap\ker(L)$ on $X$.
Then  by \cite[Lem.~1.2]{Gre:87} for $\lambda \in \rho(D_0)$ the restriction $L|_{\ker(\lambda - D_m)}:\ker(\lambda-D_m)\to\dX$ is invertible with bounded inverse
\[
L^{D_m}_\lambda := \bigl(L|_{\ker(\lambda - D_m)}\bigr)^{-1} \colon \dX
\rightarrow \ker(\lambda - D_m) \subseteq X,
\]
which we call the abstract Dirichlet operator associated to $\lambda$ and $D_m$. Note that $L^{D_m}_\lambda=L_0^{D_m-\lambda}$, that is
$L^{D_m}_\lambda x = f$ gives the unique solution of the abstract Dirichlet problem
\begin{equation*}
\begin{cases}
D_m f = \lambda f, \\
\ Lf = x .
\end{cases}
\end{equation*}
If $D_m=A_m$ we will simply write $L_\lambda:=L_\lambda^{A_m}$.

\smallbreak
Next, for a relatively $D_0$-bounded feedback operator $F:D(F)\subset X\to\dX$
we introduce the associated generalized abstract Dirichlet-to-Neumann operator $N^{D_m, F}_\lambda:D(N^{D_m, F}_\lambda)\subset\dX\to\dX$ defined by 
\begin{equation*}
N^{D_m, F}_\lambda x:=FL^{D_m}_\lambda x,
\quad
D\bigl(N^{D_m, F}_\lambda\bigr) := \bigl\{x \in \dX : L^{D_m}_\lambda x \in D(F) \bigr\}.
\end{equation*}
If $\lambda = 0$ we simply write $N^{D_m, F}:=N^{D_m, F}_0$. If in addition $F=B$ we put $N^{D_m}:=N^{D_m, B}_0$ and $N^F:=N^{A_m, F}_0$ in case $D_m=A_m$. Finally, as before we set $N:=N^{A_m, B}_0$.

\smallbreak
To proceed we need the following domain inclusions where $B,B_0:D(B)\subset X\to\dX$ are relatively $A_0$-bounded and $C:D(C)\subset \dX \rightarrow \dX$.

\begin{lem}\label{lem:inc-dom}
The following assertions hold true.
\begin{enumerate}[(i)]
\item If $C$ is relatively $N^{B_0}$-bounded, then $D(B_0) \subseteq D(CL)$. 
\item If $N^{B_0}$ is relatively $C$-bounded, then $D(A_m) \cap D(CL) \subseteq D(B_0)$. 
\end{enumerate}
\end{lem}

\begin{proof}
(i). Recall that $L_0:\dX\to\ker(A_m)$ is bijective with inverse $L$. Hence, using the first decomposition in \eqref{eq:split-X} we conclude
\begin{align*}
LD(B_0)
&=L\,\bigl((X_0\oplus\ker(A_m))\cap D(B_0)\bigr)\\
&=L\bigl(\ker(A_m)\cap D(B_0)\bigr)\\
&=L_0^{-1}\bigl(\ker(A_m)\cap D(B_0)\bigr)\\
&\subseteq D(N^{B_0})
\subseteq D(C).
\end{align*}
This implies the claim.

\smallskip
(ii). By assumption, we have
\begin{equation*}
L\,D(CL)\subseteq D(C)\subseteq D\bigl(N^{B_0}\bigr).
\end{equation*}
This implies 
\begin{equation*}
L_0L\,D(CL)\subseteq L_0D\bigl(N^{B_0}\bigr)\subseteq D(B_0).
\end{equation*}
On the other hand, $(\Id-L_0L)D(A_m)=D(A_0)\subseteq D(B_0)$. Summing up this gives the desired inclusion.
\end{proof}

Note that in part~(ii) of the previous result we cannot expect the inclusion $D(CL)\subset D(B_0)$ since always $X_0=\ker(L)\subset D(CL)$ holds.

\smallbreak
We now return to the decomposition $B=B_0+CL$ from \eqref{split-B} and consider for a relatively $A_m$-bounded perturbation $P:D(P)\subset X\to X$ the operator $(A+P)^B:D((A+B)^P)\subseteq X\to X$ given by
\begin{equation}\label{eq:A+P^B}
\begin{aligned}
(A+P)^B&\,\subseteq A_m+P, \\ 
D\bigl((A+P)^B\bigr)&:= \bigl\{ f \in D(A_m) \cap D(B_0)\cap D(CL) : LA_mf+Pf = B_0f +CL f\bigr\} . 
\end{aligned}
\end{equation}

Next we assume that $C$ is relatively $N^{B_0}=B_0L_0^{A_m}$-bounded of bound $0$. Note that by the previous lemma part~(i) this implies that $D(B)=D(B_0)\cap D(CL)=D(B_0)$. 

\begin{thm}\label{Stoerung}
Let $P \colon D(P) \subset X \rightarrow X$ be relatively $A_m$-bounded with
$A_0$-bound $0$ and let
$C \colon D(C) \subset \dX \rightarrow \dX$ be relatively $N^{B_0}$-bounded of bound $0$. 
Then for $B$ given by \eqref{split-B} the following statements are equivalent.
\begin{enumerate}[(a)]
\item $(A+P)^{B}$ in \eqref{eq:A+P^B} generates an analytic semigroup of angle $\alpha > 0$ on $X$.
\item $A^{B_0}$ generates an analytic semigroup of angle $\alpha > 0$ on $X$.
\item $A_0$ is sectorial of angle $\alpha > 0$ on $X$
and $N^{B_0}$ generates an analytic semigroup of angle $\alpha > 0$ on $\dX$.
\end{enumerate}
\end{thm}
 
Before giving the proof we state an analogous result where we interchange the roles of $N^{B_0}$ and $C$. That is, we assume that $N^{B_0}$ is relatively $C$-bounded of bound $0$. Note that by \autoref{lem:inc-dom}.(ii) this implies that $D(A_m)\cap D(B)= D(A_m)\cap D(B_0)\cap D(CL)=D(A_m)\cap D(CL)$.

\begin{thm}\label{Stoerung2}
Let $P \colon D(P) \subset X \rightarrow X$ be relatively $A_m$-bounded with
$A_0$-bound $0$ and let $N^{B_0}$ be relatively $C$-bounded of bound $0$ for some $C \colon D(C) \subset \dX \rightarrow \dX$.  
Then for $B$ given by \eqref{split-B} the following statements are equivalent.
\begin{enumerate}[(a)]
\item $(A+P)^{B}$ in \eqref{eq:A+P^B} generates an analytic semigroup of angle $\alpha > 0$ on $X$.
\item $A^{CL}$ generates an analytic semigroup of angle $\alpha > 0$ on $X$.
\item $A_0$ is sectorial of angle $\alpha > 0$ on $X$
and $C$ generates an analytic semigroup of angle $\alpha > 0$ on $\dX$.
\end{enumerate}
\end{thm}

To prove the previous two theorems we use a series of auxiliary results.
First we show the equivalences of (a) and (b) in case $P=0$.

\begin{lem}\label{C bounded}
Let $C \colon D(C) \subset \dX \rightarrow \dX$ be relatively $N^{B_0}$-bounded of bound $0$. 
Then the following statements are equivalent.
\begin{enumerate}[(a)]
\item $A^{B_0}$generates an analytic semigroup of angle $\alpha > 0$ on $X$.
\item $A^{B}$ generates an analytic semigroup of angle $\alpha > 0$ on $X$.
\end{enumerate}
\end{lem}

\begin{proof}
By \autoref{lem:inc-dom}.(i) the operator
\begin{equation*}
B := B_0 + CL , \quad D(B) = D(B_0)
\end{equation*}
is well-defined.
Since $D(A_0) \subset X_0$, the operators $B$ and $B_0$ coincide on $D(A_0)$. Hence,
$B$ is relatively $A_0$-bounded if and only if $B_0$ is relatively $A_0$-bounded of
bound $0$. 
Moreover, we have
\begin{equation*}
N^{B} = B L_0 = N^{B_0} + C , \quad D(N^{B}) = D(N^{B_0}). 
\end{equation*}
By \cite[Thm.~III.2.10]{EN:00} it then follows that $N^{B}$ generates an analytic semigroup 
of angle $\alpha > 0$ on $\dX$ if and only if $N^{B_0}$ does. 
The claim now follows by \autoref{Maintheorem}. 
\end{proof}

\begin{lem}\label{C Generator}
Let $N^{B_0}$ be relatively $C$-bounded of bound $0$ for some
$C \colon D(C) \subset \dX \rightarrow \dX$.  
Then the following statements are equivalent.
\begin{enumerate}[(a)]
\item $A^{CL}$ generates an analytic semigroup of angle $\alpha > 0$ on $X$.
\item  $A^B$ generates an analytic semigroup of angle $\alpha > 0$ on $X$.
\end{enumerate}
\end{lem}
\begin{proof}
Let
\begin{equation*}
B := B_0 + CL , \quad D(B) = D(B_0)\cap D(CL).
\end{equation*}
By the same reasoning as in the previous proof we conclude that
$B$ is relatively $A_0$-bounded if and only if $B_0$ is relatively $A_0$-bounded of the same
bound $0$. Moreover, by \autoref{lem:inc-dom}.(ii) we have
\begin{align*}
x\in D\bigl(N^B\bigr)
&\quad\iff\quad L_0x\in D(B)\\
&\quad\iff\quad L_0x\in D(B_0)\cap D(CL)\cap D(A_m)\\
&\quad\iff\quad L_0x\in D(CL)\cap D(A_m)\\
&\quad\iff\quad L_0x\in D(CL)\\
&\quad\iff\quad x\in L\,D(CL)\subseteq D(C).
\end{align*}
This implies
\begin{equation*}
N^B = B L_0 = N^{B_0} + C , \quad D(N^{B}) = D(C). 
\end{equation*}
By \cite[Thm.~III.2.10]{EN:00} it follows that $N^{B}$ generates an analytic semigroup 
of angle $\alpha > 0$ on $\dX$ if and only if $C$ does. 
The claim then follows by \autoref{Maintheorem}. 
\end{proof}

Next we study how Dirichlet operators behave under perturbations.

\begin{lem}\label{Lemma 2}
Let $P \colon D(P) \subset X \rightarrow X$ be a relatively $A_m$-bounded perturbation. Then for $\lambda \in \rho(A_0) \cap \rho(A_0+P)$ the Dirichlet operator
$L^{A_m+P}_\lambda\in\sL(\dX,X)$ exists and satisfies 
\begin{equation}\label{eq:id-pert-L_lambda}
L^{A_m+P}_\lambda - L^{A_m}_\lambda 
= R(\lambda,A_0+P)PL^{A_m}_\lambda
= R(\lambda,A_0)PL^{A_m+P}_\lambda .
\end{equation}
\end{lem}
\begin{proof}
Let $[D(A_m)]:=(D(A_m),\|\cdot\|_{A_m})$ for the graph norm $\|\cdot\|_{A_m}:=\|\cdot\|_X+\|A_m\cdot\|_X$.
Then $P \colon [D(A_m)] \rightarrow X$ and $L^{A_m}_\lambda \colon \partial X \rightarrow [D(A_m)]$ are bounded, hence $PL^{A_m}_\lambda:\dX\to X$ is bounded as well.
This implies that
\begin{equation*}
T: = L^{A_m}_\lambda + R(\lambda,A_0 + P)PL^{A_m}_\lambda\in\sL(\dX,X) .
\end{equation*}
Since 
\begin{align*}
(A_m+P - \lambda) T x &=  (A_m+P - \lambda) L^{A_m}_\lambda x 
+ (A_m+P - \lambda) R(\lambda,A_0+P)PL^{A_m}_\lambda \\
&= P L^{A_m}_\lambda x - P L^{A_m}_\lambda x = 0 ,
\end{align*}
we have $\rg(T) \subseteq \ker(\lambda - A_m -P)$. 
Moreover, from
\begin{equation*}
\rg\bigl(R(\lambda,A_0+P)PL^{A_m}_\lambda\bigr) \subset 
D(A_0+P) = D(A_0) \subset \ker(L)
\end{equation*}
it follows that $LT x = LL^{A_m}_\lambda x = x$.
Hence, $L|_{\ker(\lambda - A_m - P)}$ is surjective
with right-inverse $T$. 
Since $\ker(\lambda - A_m - P) \cap X_0 
\subset \ker(\lambda - A_0 - P) = \{ 0 \}$ we conclude
that $L|_{\ker(\lambda - A_m - P)}$ is injective as well.
This implies that it is invertible with inverse $L^{A_m+P}_\lambda=T$ and proves the first identity in \eqref{eq:id-pert-L_lambda}. The second one follows by changing the roles of $A_m$ and $A_m+P$.
\end{proof}

Next we consider perturbations of Dirichlet-to-Neumann operators.

\begin{prop}\label{Stoerungssatz für D-N}
Let $P \colon D(P) \subset X \rightarrow X$ be a relatively $A_m$-bounded perturbation. Then for $\lambda \in \rho(A_0) \cap \rho(A_0+P)$ the perturbed Dirichlet-to-Neumann operator $N^{A_m+P}_\lambda$ exists, $D(N^{A_m}_\lambda)=D(N^{A_m+P}_\lambda)$ and the difference $N^{A_m}_\lambda - N^{A_m+P}_\lambda$ is bounded. 
\end{prop}

\begin{proof} 
Since 
\begin{equation*}
\rg\bigl(R(\lambda,A_0)(A_m-\lambda)L^{A_m+P}_\lambda\bigr) \subset D(A_0) \subset D(B),
\end{equation*}
by \autoref{Lemma 2} it follows 
that $D(N^{A_m}_\lambda) = D(N^{A_m+P}_\lambda)$. Moreover,
from \eqref{eq:id-pert-L_lambda} we conclude
\begin{equation*} 
N^{A_m}_\lambda - N^{A_m+P}_\lambda 
= BL^{A_m}_\lambda - BL^{A_m+P}_\lambda 
\supseteq -BR(\lambda,A_0)PL^{A_m+P}_\lambda\in\sL(\dX).
\qedhere
\end{equation*}
\end{proof}

To conclude the proofs of \autoref{Stoerung} and \autoref{Stoerung2}, we need one further result. It shows that the assertion~(a) in both results is independent under the perturbation $P$. 

\begin{lem}\label{Stoerung3}
Let $P \colon D(P) \subset X \rightarrow X$ relatively $A_m$-bounded with
$A_0$-bound $0$.   
Then the following statements are equivalent.
\begin{enumerate}[(a)]
\item $A^B$ generates an analytic semigroup of angle $\alpha > 0$ on $X$.
\item $(A+P)^B$ generates an analytic semigroup of angle $\alpha > 0$ on $X$.
\end{enumerate}
\end{lem}

\begin{proof}
Since $A_0$ is a weak Hille--Yosida operator and $P$ is relatively $A_0$-bounded of 
bound $0$, by \cite[Lem.~III.2.6]{EN:00} there exists a $\lambda \in \rho(A_0) \cap \rho(A_0+P)$ and
$A_0- \lambda$, $A_0 + P - \lambda$ are again weak Hille--Yosida operators. 
Since $B$ is relatively $A_0$-bounded of bound $0$ a simple computation shows that it is also
relatively $(A_0 - \lambda)$- and $(A_0 + P - \lambda)$-bounded of bound $0$.
Moreover, by \autoref{Lemma 2} the operators $L_0^{A_m - \lambda}$ and $L_0^{A_m + P - \lambda}$
exist and are bounded. 
Hence, $A_0 - \lambda$ and $A_0 + P - \lambda$ both satisfy \autoref{masu}.

Next we check the conditions in \autoref{Maintheorem}.
By \cite[Lem.~III.2.6]{EN:00} the operator $A_0 - \lambda$ is sectorial of angle $\alpha >0$ on $X$ if and only if $A_0 + P - \lambda$ is. 
Moreover, by \autoref{Stoerungssatz für D-N} $N^{A_m- \lambda}$ generates an analytic 
semigroup of angle $\alpha > 0$ if and only if $N^{A_m + P - \lambda}$ does. 
Applying \autoref{Maintheorem} to $A_0 - \lambda$, $N^{A_m - \lambda}$ and
$A_0 + P - \lambda$, $N^{A_m + P - \lambda}$, respectively, the claim follows. 
\end{proof}

\begin{proof}[Proof of \autoref{Stoerung} and \autoref{Stoerung2}]
By \autoref{Stoerung3} assertion~(a) is independent of $P$ while by \autoref{C bounded} and \autoref{C Generator}, respectively, for $P=0$ it is equivalent to (b). Since the equivalence of (b) and (c) follows \autoref{Maintheorem} the proof is complete.
\end{proof}

\section{Examples}\label{Expl}

\subsection{Second Order Differential Operators on $\mathbf{\rC([0,1],\C^n)}$.}

For $n\in\N$ consider functions $a_i \in\rC[0,1]\cap\rC^1(0,1)$, $1\le i\le n$, being strictly positive on $(0,1)$ such that $\frac{1}{a_i} \in L^1[0,1]$. Let
$a := \diag(a_1,\dots, a_n)$ and $b, c \in \rC([0,1],\MnC)$.
Moreover, define the maximal operator $A_m:D(A_m)\subset\rC([0,1],\C^n)\to\rC([0,1],\C^n)$ by
\begin{equation*}
A_m := a f''+ bf'
+ cf, \quad D(A_m) := \bigl\{ f \in \rC([0,1],\C^n)\cap \rC^2((0,1),\C^n) \colon A_m f \in \rC([0,1],\C^n) \bigr\}
\end{equation*}
and take $B \in \mathcal{L}(C^1([0,1],\C^{n}),\C^{2n})$. 

\begin{cor}\label{exa1} We have
$D(A_m)\subset\rC^1([0,1],\C^n) = D(B)$ and 
\begin{equation*}
A\subseteq A_m, \quad
D(A) = \left\{ f \in D(A_m) \colon \begin{psmallmatrix} (A_m f)(0) \\ (A_m f)(1) \end{psmallmatrix}
= Bf \right\}
\end{equation*}
generates a compact and analytic semigroup of angle $\frac{\pi}{2}$ on
$\rC([0,1],\C^n)$. 
\end{cor}

\begin{proof}
We consider $X := \rC([0,1],\C^n) = \rC[0,1] \times \dots
\times \rC[0,1]$ equipped with the norm
$\| f \|_{1,\infty} := \|f_1\|_\infty + \dots + \|f_n\|_\infty$, $\dX := \C^{2n}$ and define $L\in\sL(X,\dX)$ by
$Lf := \binom{f(0)}{f(1)}$. Then as in \cite[Cor. 4.1 Step (iii)]{EF:05} it follows that $D(A_m)\subset D(B)$, hence $A$ coincides with the operator defined in \eqref{eq:W-BC}. Since 
\begin{equation*}
Pf := bf' + c f , \quad 
D(P) :=\rC^1\bigl([0,1], \C^n\bigr)
\end{equation*}
is a relatively $A_m$-bounded with $A_0$-bound $0$ (see Step 4 below), we assume by \autoref{Stoerung} without loss of generality that $b=c=0$. 

Next we verify \autoref{masu} and the hypotheses of \autoref{Maintheorem}.

\setcounter{stpi}{0}
\begin{stepi}\label{Step_1_a}
The abstract Dirichlet operator $L_0 \in \mathcal{L}(\partial X, X)$ exists. 
\end{stepi}
\begin{proof}
We have $\ker(A_m) = \text{lin}\{\epsilon_0, \epsilon_1\}$ for
\begin{equation*}
\epsilon_0(s) := 1-s \quad\text{and}\quad \epsilon_1(s) := s,\ s\in [0,1].
\end{equation*}
A simple calculation then shows that $L_0 := (L|_{\ker(A_m)})^{-1}\in\sL(\dX,X)$ is given by
\begin{equation*}
L_0 \begin{psmallmatrix}
x_1 \\
\vdots \\
x_{2n}
\end{psmallmatrix} = 
\epsilon_0\cdot
\begin{psmallmatrix}
x_1 \\
\vdots \\
x_{n}
\end{psmallmatrix}
+
\epsilon_1\cdot
\begin{psmallmatrix}
x_{n+1} \\
\vdots \\
x_{2n}
\end{psmallmatrix}.
\end{equation*}
\end{proof}

\begin{stepi}\label{Step_2_a}
The operator $A_0$ on $X$ is sectorial of angle $\frac{\pi}{2}$ and has compact resolvent.
\end{stepi}
\begin{proof}
Let $A_i := a_i\cdot\frac{d^2}{ds^2}$ with domain
$D(A_i) := \{ g \in\rC[0,1]\cap\rC^2(0,1) \colon a_i\cdot g'' \in \rC[0,1] \}$ for $1\le i\le n$.
Then
\begin{equation*}
R(\lambda,A_0) = \diag\bigl(
R(\lambda,A_1),\ldots, R(\lambda,A_n)
\bigr).
\end{equation*}
Since by \cite[Cor. 4.1. Step (ii)]{EF:05} all $A_i$ are sectorial of angle $\frac{\pi}{2}$ and have compact resolvents on $\rC[0,1]$, 
the claim follows.
\end{proof}

\begin{stepi}
The maximal operator $A_m$ is densely defined and closed. 
\end{stepi}
\begin{proof}
Since $\rC^2([0,1],\C^n) \subset D(A_m)$, $A_m$ is densely defined. By \autoref{Step_1_a}, \autoref{Step_2_a}
and \cite[Lem. 3.2]{EF:05} it follows that $A_m$ is closed. 
\end{proof}

\begin{stepi}
The feedback operator $B$ is relatively $A_0$-bounded of bound $0$.
\end{stepi}
\begin{proof}
Since $D(B) =\rC^1([0,1],\C^n)$ it suffices to show
that the first derivative with domain $\rC^1([0,1],\C^n)$ is relatively $A_0$-bounded with bound $0$. Let $f\in D(A_0)$. Then by \cite[Cor. 4.1. Step (iii)]{EF:05} it follows
that for all $\epsilon > 0$ there exists a 
constant $C_\epsilon > 0$ such that
\begin{align*}
\| f' \|_{1,\infty} 
&\leq \epsilon\cdot \| A_1 f_1 \|_\infty  + \dots + \epsilon\cdot \| A_n f_n \|_\infty + C_\epsilon\cdot \| f_1 \|_\infty + \dots + C_\epsilon\cdot \| f_n \|_\infty \\
&= \epsilon\cdot \| A_0 f \|_{1,\infty} + C_\epsilon\cdot  \| f \|_{1,\infty}.
\qedhere
\end{align*}
\end{proof}

\begin{stepi}
The Dirichlet-to-Neumann operator $N$ generates an analytic, compact semigroup of angle $\frac{\pi}{2}$ on $\dX$. 
\end{stepi}
\begin{proof}
Since the boundary space $\dX$ is finite dimensional, 
$N$ is bounded. Hence $N$ generates an analytic, compact semigroup of angle $\frac{\pi}{2}$ on $\dX$.
\end{proof}

Summing up, by \autoref{Maintheorem} and \autoref{kompakt} the claim follows completing the proof.
\end{proof}

\begin{rem}
\autoref{exa1} generalizes \cite[Cor. 4.1]{EF:05} to arbitrary $n \in \N$.
\end{rem}

We give a particular choice for the operator $B$.

\begin{cor}
For $M_i, N_i \in \mathrm{M}_{2n\times n}(\C)$, $i=0,1$,
the operator
\begin{equation*}
A\subseteq A_m, \ D(A) = \left\{ f \in D(A_m) \colon
\begin{pmatrix}
(A_mf)(0) \\
(A_mf)(1)
\end{pmatrix}
=
M_0 f'(0) + M_1 f'(1) + N_0 f(0) + N_1 f(1)
\right\}
\end{equation*}
generates a compact and analytic semigroup of angle
$\frac{\pi}{2}$ on $\rC([0,1],\C^n)$. 
\end{cor}

We remark that second order differential operators on spaces of functions $f:[0,1]\to\C^n$ can be used to describe diffusion- and waves on networks. For some recent results in the $\rL^p$-context for operators with generalized Robin-type boundary conditions we refer to \cite{EKF:17}.

\subsection{Banach Space-Valued Second Derivative}

We associate to an arbitrary Banach space $Y$ the Banach space
$X :=\rC([0,1],Y)$ of all continuous functions on $[0,1]$ with values in 
$Y$ equipped with the sup-norm. Moreover, we take $P \in \mathcal{L}(\rC^1([0,1],Y), X))$,
$\Phi \in \mathcal{L}(X,Y^2)$
and an operator $(\sC,D(\sC))$ on $Y^2$. Then the following holds.

\begin{cor}
The operator $\sC$ generates an analytic semigroups of angle $\alpha \in (0,\frac{\pi}{2}]$ on $Y^2$
if and only if the operator
\begin{align*}
Af &:= f'' + P f ,\\ D(A) &:= \left\{ f \in\rC^2([0,1],Y) \colon 
\tbinom{f(0)}{f(1)}\in D(\sC), 
\tbinom{f''(0) + Pf(0)}{f''(1) + Pf(1)}
=
\Phi f+ \sC\tbinom{f(0)}{f(1)} \right\}
\end{align*}
generates an analytic semigroup of angle $\alpha \in (0,\frac{\pi}{2}]$ on $X$. 
\end{cor}

\begin{proof}
We consider $\dX := Y^2$ and define $L\in\sL(X,\dX)$ by
$Lf := \binom{f(0)}{f(1)}$. Moreover, define
\begin{equation*}
A_m:D(A_m)\subseteq X\to X,\quad
A_m f := f'' + Pf , \quad D(A_m) =\rC^2([0,1],Y) 
\end{equation*}
and 
\begin{equation*}
B:D(B)\subseteq X\to\dX,\quad
B f :=\Phi f+\sC Lf, \quad D(B):=\bigl\{f\in X:\tbinom{f(0)}{f(1)}\in D(\sC)\bigr\}.
\end{equation*}
Then $A$ coincides with the operator given by \eqref{eq:W-BC}. Since $P$
is a relatively $A_m$-bounded of $A_m$-bound $0$ and $\Phi \in \mathcal{L}(X,\dX)$,
by \autoref{Stoerung2} it suffices to verify the \autoref{masu} and that $A_0$ is sectorial of angle $\alpha > 0$.

\setcounter{stp}{0}
\begin{stepii}\label{Step 1. b}
The abstract Dirichlet operator $L_0 \in \mathcal{L}(\partial X, X)$ exists. 
\end{stepii}
\begin{proof}
As in \autoref{Step_1_a} of the proof of \autoref{exa1} we have
$\ker(A_m) = \{ \epsilon_0 y_0 + \epsilon_1 y_1 \colon y_0,y_1 \in Y \} $ for
\begin{equation*}
\epsilon_0(s) := 1-s \quad\text{and}\quad \epsilon_1(s) := s,\ s\in [0,1].
\end{equation*}
Moreover, $L_0 := (L|_{\ker(A_m)})^{-1}\in\sL(\dX,X)$ is given by
\begin{equation*}
L_0 
\tbinom{y_0}{y_1} = 
\epsilon_0\cdot
y_0 + \epsilon_1\cdot y_1.
\qedhere
\end{equation*}
\end{proof}

\begin{stepii}\label{Step 2. b}
The operator $A_0$ on $X$ is sectorial of angle $\frac{\pi}{2}$.
\end{stepii}
\begin{proof}
This follows as in the proof of \cite[Thm VI. 4.1]{EN:00}.
\end{proof}
 
\begin{stepii} 
The maximal operator $A_m$ is densely defined and closed. 
\end{stepii}
\begin{proof}
Since $\rC^2([0,1],Y) \subset D(A_m)$, $A_m$ is densely defined. By \autoref{Step 1. b}, \autoref{Step 2. b}
and \cite[Lem. 3.2]{EF:05} it follows that $A_m$ is closed. 
\end{proof}

\begin{stepii}
The feedback operator $B$ is relatively $A_0$-bounded of bound $0$.
\end{stepii}
\begin{proof}
For $f \in D(A_0) \subset X_0$ we have $Bf = \Phi f$. Since $\Phi$ is bounded, this implies the claim.
\end{proof}

Summing up, by \autoref{Maintheorem} the claim follows completing the proof.
\end{proof}

\subsection{Perturbations of the Laplacian on $\mathbf{\COq}$ with generalized Wentzell boundary conditions}\label{LgW}

In this subsection we complement the example from the introduction concerning the Laplacian on $\COq$ with generalized Wentzell boundary conditions, see also \cite{Eng:03}.

To this end we consider a bounded domain $\Omega \subset \R^n$ with $\rC^{\infty}$-boundary $\partial \Omega$ and take an operator $P\in\sL(\rC^1(\overline{\Omega}),\COq)$ (e.g. a first-order differential operator). Then we define the perturbed Laplacian 
$A:D(A)\subset\COq\to\COq$ with generalized Wentzell boundary conditions  by $Af:=\Delta_mf+Pf$ for
\begin{equation}\label{eq:bc-W-Lap-LB} 
f\in D(A)
\quad:\iff\quad
(\Deltam f+Pf)\big|_{\partial\Omega}=\beta\cdot\tddn\; f+\gamma\cdot f\big|_{\dO}
+q\cdot\DLB f|_{\dO},
\end{equation}
cf. also \cite[(1.2), (3.3)]{FGGR:10}. Here $\beta<0$, $\gamma\in\CdO$, $q\ge0$ and $\DLB:D(\DLB)\subset\CdO\to\CdO$ denotes the Laplace--Beltrami operator. In case $P=0$, $q=0$ this just gives the operator $A$ from the introduction. As we will see below for $q>0$ the Laplace--Beltrami operator will dominate the dynamic on the boundary $\dX$.\footnote{The discussion of this case has been inspired by a discussion with J. Goldstein.} However, in this case essentially the same generation result holds as for $q=0$.

\begin{cor}\label{cor:D_LB}
For all $q>0$ the operator $A\subseteq\Delta_m+P$ with domain given in \eqref{eq:bc-W-Lap-LB} generates a compact and analytic semigroup of angle $\frac{\pi}2$.
\end{cor}

\begin{proof} Without loss of generality we assume that $\beta=1$.
To fit the operator $A$ into our setting we define $X:=\COq$, $\dX:=\CdO$ and the trace  $L\in\sL(X,\dX)$, $Lf:=f|_{\dO}$. Then we consider $A_m:=\Delta_m:D(\Delta_m)\subset X\to X$ and $B_0:=\frac{\partial}{\partial n}:D(\frac{\partial}{\partial n})\subset X\to\dX$ as in \cite{Eng:03} and put $C:=q\cdot\DLB+M_\gamma:D(\DLB)\subset\dX\to\dX$ and $B:=B_0+CL$ as in \eqref{split-B}.

Then by \cite[Thm.~6.1.3]{ABHN:01}, $A_0=\Delta_0$ is sectorial of angle $\frac\pi2$ and by \cite[(1.9)]{Eng:03} and \cite[Prop.~II.4.25]{EN:00} has compact resolvent. Moreover, $C$ generates a compact analytic semigroup of angle $\frac\pi2$. Let $W:=(-\DLB)^{\frac12}$. Then by the proof of \cite[Thm.~2.1]{Eng:03} there exists a relatively $W$-bounded perturbation $Q:D(Q)\subset\dX\to\dX$ such that $N^{B_0}=B_0L^{A_m}_0=-W+Q$. This implies that $N^{B_0}$ is relatively $W$-bounded and by \cite[Thm.~6.10]{Paz:83} it follows that $N^{B_0}$ is relatively $C$-bounded of bound $0$. Hence, by \autoref{Stoerung2}, $(A+P)^B$ generates an analytic semigroup of angle $\frac\pi2$. Compactness of this semigroup follows by \autoref{kompakt}.
\end{proof}

We remark that \autoref{cor:D_LB} confirms the conjecture $\theta_\infty=\frac{\pi}{2}$ in \cite[Sect.~5]{FGGR:10} for $a(x)\equiv\Id$ and constant $\beta<0$.

\subsection{Uniformly Elliptic Operators on $\mathbf{\rC(\overline{\Omega})}$}

We consider a uniformly elliptic second-order differential operator
with generalized Wentzell boundary conditions on $\rC(\overline{\Omega})$ for a 
bounded domain $\Omega \subset \R^n$ with $\rC^{\infty}$-boundary $\partial \Omega$.
To this end, we first take real-valued functions
\begin{equation*}
a_{jk} = a_{kj} \in\rC^\infty(\overline{\Omega})
, \quad a_j, a_0, b_0 \in \rC(\overline{\Omega}), \quad 1 \leq j,k \leq n
\end{equation*}
satisfying the uniform ellipticity condition
\begin{equation*}
\sum_{j,k = 1}^n a_{jk}(x)\cdot \xi_j \xi_k \geq c\cdot \| \xi \|^2 \quad \text{ for all } x \in \overline{\Omega}, \ 
\xi = (\xi_1, \dots, \xi_n) \in \R^n
\end{equation*}
and some fixed $c > 0$. Then we define the maximal operator $A_m:D(A_m)\subseteq\COq\to\COq$ in divergence form by
\begin{align*}
A_m f &:= \sum_{j = 1}^n \partial_j \Bigl( \sum_{k = 1}^n a_{jk} \partial_k f \Bigr) 
+ \sum_{k = 1}^n a_k \partial_k f + a_0 f, \\
D(A_m) &:= \biggl\{ f \in \bigcap_{p \geq 1} W^{2,p}_{\text{loc}}(\Omega) \colon A_m f \in \rC(\overline{\Omega}) \biggr\},
\end{align*}
and the feedback operator $B:D(B)\subseteq\COq\to\CdO$ by
\begin{equation*}
B:= - \sum_{j,k = 1}^n a_{jk}\nu_j L \partial_k + b_0 L ,\quad
D(B) := \biggl\{ f \in \bigcap_{p \geq 1} W^{2,p}_{\text{loc}}(\Omega) \colon Bf \in \rC(\partial \Omega) \biggr\},
\end{equation*}
where $L\in\sL(\COq,\CdO)$, $Lf := f|_{\partial \Omega}$ denotes the trace operator. 

\begin{cor}
The operator $A:D(A)\subseteq\COq\to\COq$ given by
\begin{equation*}
A \subseteq A_m,\quad
D(A) := \bigl\{ f \in D(A_m) \cap D(B) \colon LA_m f = Bf \bigr\} 
\end{equation*}
generates a compact and analytic semigroup on $\rC(\overline{\Omega})$. 
\end{cor}
\begin{proof}
Let $X:=\COq$, $\dX:=\CdO$ and
define the maximal operator $\tilde A_m:D(\tilde A_m)\subseteq X\to X$ by
\begin{equation*}
\tilde{A}_m := \sum_{j = 1}^n \partial_j \Bigl( \sum_{k = 1}^n a_{jk} \partial_k \Bigr), \quad
D(\tilde{A}_m) := %\biggl\{ f \in \bigcap_{p \geq 1} W^{2,p}_{\text{loc}}(\Omega) \colon \tilde{A}_m f \in \rC(\overline{\Omega}) \biggr\} = 
D(A_m),
\end{equation*}
and the feedback operator $\tilde B:D(\tilde B)\subseteq\COq\to\CdO$ by
\begin{equation*}
\tilde{B} := - \sum_{j,k = 1}^n a_{jk}\nu_j L \partial_k , \qquad
D(\tilde{B}) := \biggl\{ f \in \bigcap_{p \geq 1} W^{2,p}_{\text{loc}}(\Omega) \colon \tilde{B}f \in \rC(\partial \Omega) \biggr\}.
\end{equation*}
Then by \cite[Cor. 4.5]{EF:05} it follows
that the operator $\tilde A:D(\tilde A)\subseteq X\to X$ with generalized Wentzell boundary conditions given by
\begin{equation*}
\tilde{A}\subseteq\tilde{A}_m , \quad
D(\tilde{A}) := \bigl\{ f \in D(\tilde{A}_m) \cap D(\tilde{B}) \colon L\tilde{A}_m f = \tilde{B}f \bigr\} 
\end{equation*}
generates a compact and analytic semigroup on $X$.
Let $P f := \sum_{j = 1}^n a_j \partial_j f + a_0 f$ and $C f := b_0 f$. 
Then $P$ is relatively $A_m$-bounded with bound $0$ and $C \in \mathcal{L}(\dX)$,
hence the claim follows from \autoref{Stoerung}. 
\end{proof}

\begin{rem}
This result generalizes \cite[Cor. 4.5]{EF:05} and via \autoref{Maintheorem} also the main theorem in \cite{Esc:94}.
Moreover, it shows that the angle of the analytic semigroup generated by $A$ only depends on the matrix $(a_{jk})_{n\times n}$. 
\end{rem}

\section{Conclusion}
Our abstract approach allows to decompose an operator $A$ with generalized Wentzell boundary conditions into an operator $A_0$ with (much simpler) abstract Dirichlet boundary conditions and the associated abstract Dirichlet-to-Neumann operator $N$. In particular  we prove, under a weak resolvent condition on $A_0$, that
\[
\begin{aligned}
A&\text{ generates an analytic semigroup}\\
 &\text{ of angle $\alpha>0$}
\end{aligned}
\Biggr\}
\quad\iff\quad
\left\{
\begin{aligned}
A_0&\text{ is sectorial of angle $\alpha>0$, and}\\
N&\text{ generates an analytic semigroup}\\
  &\text{ of angle $\alpha>0$,}
 \end{aligned}
\right.
\]
cf. \autoref{Maintheorem}.
This equivalence is new and shows the sharpness of our approach. Moreover, while being very general, our theory applied to concrete examples (where typically $A_0$ is well-understood and sectorial of angle $\frac\pi2$) gives new or improves known generation results, see \autoref{Expl}.

%\bibliography{/home/kje/Dokumente/TeX/BiB-files/book}{}
%%\bibliography{book}{}
%\bibliographystyle{meinstil}

\newcommand{\etalchar}[1]{$^{#1}$}

\bigskip
\emph{Tim Binz}, University of Tübingen, Department of Mathematics, Auf der Morgenstelle 10, D-72076 Tübingen, Germany,
\texttt{tibi@fa.uni-tuebingen.de}

\smallskip
\emph{Klaus-Jochen Engel}, University of L'Aquila, Department of Information Engineering, Computer Science and Mathematics, Via Vetoio,  I-67100 L'Aquila -- Coppito (AQ), Italy,
\texttt{klaus.engel@univaq.it}

\end{document}